\documentclass[a4paper,12pt]{article}
\usepackage{setspace,url, graphicx, float, subcaption, caption,amsmath}
\usepackage[top=2.5cm, bottom=2.5cm, right=2.5cm, left=2.5cm]{geometry}
\usepackage[hidelinks]{hyperref}
\usepackage{fancyhdr}
\usepackage{tabu}
\usepackage{ amssymb }
\usepackage{amsthm}
\usepackage{imakeidx}
\usepackage{titletoc}
\usepackage{xcolor}

\newcommand{\Cay}{\mbox{Cay}}
\newcommand{\Aut}{\mbox{Aut}}

\newtheorem{thm}{Theorem}[subsection]
\newtheorem{thms}{Theorem}[section]
\newtheorem{dfn}{Definition}[subsection]

\newtheorem{lems}{Lemma}[section]

\begin{document}
\title{Further Applications of Schur Rings to Produce GRRs for Dihedral Groups.}
\author{Jonathan Ebejer, Josef Lauri \\University of Malta}

\maketitle

\textbf{Abstract:} \emph{This note is a continuation of postgraduate thesis research carried out by the first author under the supervision of the second author at the University of Malta. In that research we took a look at several results relating Schur rings to sufficient conditions for GRRs and then applied those results to produce numerical methods for constructing trivalent GRRs for dihedral groups very quickly.}

\emph{Here we seek to slightly alter the approach used in the original research to produce another numerical method for producing other GRRs for dihedral groups. Moreover, this will hopefully serve as companion literature for the original research when published.}
	
\section{Foundations}

The background needed to understand the proof below would be quite familiar to abstract algebraists and algebraic graph theorist and is given a more thorough treatment in the original postgraduate research \cite{Eb} and the work which was uploaded to arXiv \cite{Eb&La}. Still, we make a few critical definitions below and cite theorems and lemmas which will be needed to prove the main theorem of this work. The proofs for these lemmas and theorems have been omitted but are available at the cited sources.

Let us start with a few definitions:

\vspace{4mm}

\begin{dfn}[Cayley graphs]
The \textbf{Cayley graph} of a group $\Gamma$ is the graph, denoted by $\Cay(\Gamma, S)$, whose vertex-set is $\Gamma$ and two vertices $u$ and $v$ are adjacent if $v=us$ where $s\in S$ and $S$ is a subset of $\Gamma$ such that $1\not\in S$, $S$ generates $\Gamma$, and $S^{-1}=S$. We call the set $S$ the \textbf{connecting set} of the Cayley graph.
\end{dfn}

\vspace{2mm}

$\mathbf{Schur\ Rings}$ A subring $\mathcal{S}$ of the group ring $\mathbb{Z}[\Gamma]$ is called a Schur ring or an $\mathcal{S}$-ring over $\Gamma$, of rank $r$ if the following conditions hold:
\begin{itemize}
	\item $\mathcal{S}$ is closed under addition and multiplication including multiplication by elements of  $\mathbb{Z}$ from the left (i.e. $\mathcal{S}$ is a $\mathbb{Z}[\Gamma]$-module).
	\item Simple quantities $\overline{B}_{0},\overline{B}_{1},...,\overline{B}_{r-1}$ exist in $\mathcal{S}$ such that every element $\overline{C} \in \mathcal{S}$ has a unique representation:
	\[\overline{C}=\sum_{i=0}^{r-1}\beta^{i}\overline{B}_{i}\] where $\beta^i \geq 0$.
	\item $\overline{B}_{0}=\overline{1}$.
	\item $\displaystyle \sum_{i=0}^{r-1}\overline{B}_{i}=\overline{\Gamma}$, that is, $\lbrace B_{0},B_{1},...,B_{r-1}\rbrace$ is a partition of $\Gamma$.
	\item For every $i\in \lbrace 0,1,2,...,r-1 \rbrace$ there exists a $j \in \lbrace 0,1,2,...,r-1 \rbrace$ such that $\overline{B}_{j}=\overline{B}^{-1}_{i} (=\overline{\lbrace b^{-1} : b \in B_{i}} \rbrace )$.
\end{itemize}

We call the set of simple quantities $\overline{B}_{0},\overline{B}_{1},\ldots,\overline{B}_{r-1}$the \textbf{basis} of the Schur ring and we denote it by ${\mathcal B}[{\mathcal S}]$. Each simple quantity $\overline{B}_i$ of the basis is referred to as  a \textbf{basic element} of the Schur ring. Sometimes we need to refer to the set $B_i$ which we call a \textbf{basic set}. The Cayley graphs formed used the basic sets of a Schur ring as the connecting sets are referred to as the \textbf{basic Cayley graphs} of the Schur ring.

Let $C$ be a subset of the group $\Gamma$. Then $\langle\langle C \rangle\rangle$ is the smallest (coarsest) Schur ring of $\Gamma$ containing $\overline{C}$.

\vspace{4mm}

$\mathbf{ Structure\ Constants}$ Let $\overline{B}_i$ and $\overline{B}_j$ be two basis elements of an r-rank Schur ring $\mathcal{S}$. For all values $i, j, k \in [r]$ there exist non-negative integers $\beta_{ij}^{k}$ called structure constants, such that
\[\overline{B}_{i}\cdot\overline{B}_{j}=\sum_{k=1}^{r} \beta_{ij}^{k}\overline{B}_{k}\]

\vspace{4mm}

$\mathbf{ Automorphism\ groups\ of\ Schur\ rings}$ The automorphism group of a Schur ring is defined to be the intersection of all automorphism groups of the basic Cayley graphs of the Schur ring. For any Schur ring $\mathcal{S}$ we will denote its automorphism group as $Aut(\mathcal{S})$.

\vspace{4mm}
	
The ``largest'' Schur ring of a group $\Gamma$ is defined to be the finest partition of $\Gamma$, that is, the basic sets are all the singleton elements of $\Gamma$. The ``smallest'' Schur ring is then the coarsest partition, with basic sets $\{1\}$ and $\Gamma - 1$    

The smallest (coarsest) Schur ring has the largest automorphism group, that is, the symmetric group $S_n$, where $n$ is the order of the group. The largest Schur ring has the smallest automorphism group, that is,  the regular action of the group on itself.

\vspace{4mm}

$\mathbf{ Graphical\ Regular\ Representation}$ Let $\Gamma$ be a finite group and let $G$ be a graph. If $\Aut(G) \equiv \Gamma$ and acts regularly on $V(G)$ then we say that $G$ is a graphical regular representation (GRR)of $\Gamma$.

The question asking which finite groups have at least one GRR has been completely solved \cite{Go, He, Im, Im2, ImWa, ImWa2, NoWa, NoWa2}. Specifically:

\begin{enumerate}
	\item The only abelian groups which have a GRR are $\mathbb{Z}^n_2$ for $n\geq5$.
	
	\item Except for generalised dicyclic groups and a finite number of known groups, all non-abelian groups have a GRR.
\end{enumerate}

However, it remains a challenge to discern whether a group known to have GRRs has GRRs with specific properties, such as being trivalent. This question is the focus of this note, specifically for dihedral groups.

\vspace{2mm}

Below are some results we will be using: 

\vspace{4mm}

\begin{thm}\label{thm:aut<<C>>}
	$\Aut(\Cay(\Gamma,C))=\Aut(\langle\langle C \rangle\rangle )$ \textnormal{\cite{MuPo}}
\end{thm} 


\begin{thm}\label{thm:trivial<<C>>}
If $\langle \langle C \rangle \rangle$ is trivial then $Cay(\Gamma,C)$ is a GRR of $\Gamma$.   \textnormal{\cite{Wei}}
\end{thm}

\begin{thm}[Schur-Wielandt Principle]
Let $r=\sum z_\gamma \gamma$ be an element of a Schur ring $\mathcal{S}$ of a group $\Gamma$. Then, for any integer $k$, the sum $\sum_{z_\gamma=k} \gamma$ is also in $\mathcal{S}$. \textnormal{\cite{MuPo}}
\end{thm}   

\vspace{8mm}

\subsection{Main Results}


Our main result is the following:


\begin{thms}\label{thm:main}
Let $n$ be an odd integer greater than 5 and let $r$, $s$, and $t$ be integers less than $n$ such that the difference of any two of them is relatively prime to $n$. If $3r+s=4t \bmod n$, then $\Cay(D_n, \{ab^r, ab^s, ab^t\})$ is a GRR of $D_n$.
\end{thms}

To prove this theorem however, we will need the help of the following two lemmas. We include only the proof of the first one as the second lemma is proven in the main work on arXiv.


\begin{lems} \label{lm1}
Let $\gamma$ be an element of $\Gamma$ such that $\gamma$ has odd order $n$ and let $r$ and $t$ be two integers such that $r-t$ is relatively prime to $n$. If $\gamma^r + \gamma^t$ is an element of some $\mathcal{B}[\mathcal{S}_\Gamma]$ then  $\gamma^s + \gamma^{-s}$ is also an element of that same $\mathcal{B}[\mathcal{S}_\Gamma]$ for all integers $s$ relatively prime to $n$.
\end{lems}

\begin{proof}
First we observe that since $\mathcal{S}_\Gamma$ is closed then $(\gamma^r + \gamma^t)^n$ is in $\mathcal{S}_{\Gamma}$ whenever $(\gamma^r + \gamma^t)$ is in $\mathcal{S}_{\Gamma}$. We can use binomial expansion to get:

$(\gamma^r+\gamma^ t)^n = \displaystyle\sum_{i=0}^{n}[\binom{n}{i}\gamma^{ri}\gamma^{ t(n-i)}]$ which can be re-written as $\displaystyle\sum_{i=0}^{n}[\binom{n}{i}\gamma^{(r- t)i)}]\gamma^{tn}$. However, $\gamma$ is of order $n$ and so $\gamma^{tn} \equiv 1$. Therefore, the sum can once more be rewritten as $\displaystyle\sum_{i=0}^{n}[\binom{n}{i}\gamma^{(r- t)i)}]$.

\medskip
Since $r- t$ is relatively prime to $n$ then the summation $\displaystyle\sum_{i=0}^{n}[\binom{n}{i}\gamma^{(r- t)i)}]$ must include exactly $n$ terms. In other words every power of $\gamma$ appears in that sum once and only once. Therefore, the coefficient of every $\gamma^{(r-t)i}$ is exactly $\binom{n}{i}$. Moreover, we observe that $\binom{n}{i} = \binom{n}{n-i}$ and so we can re-write this summation as $\displaystyle\sum_{i = 0}^{\frac{n}{2}}[\binom{n}{i}(\gamma^{(r- t)i}+\gamma^{(t-r)i})]$.

\medskip
This means that for every $i \in [0, \frac{n}{2}]$ either $\gamma^{(r-t)i}$ occurs as a singleton in $\mathcal{B}[\mathcal{S}_\Gamma]$ or $\gamma^{(r-t)i}+\gamma^{(t-r)i}$ occurs in $\mathcal{B}[\mathcal{S}_\Gamma]$ due to the Schur-Wielandt principle.

However, in the cases where $i$ is relatively prime to $n$, $(r-t)i$ would also be relatively prime to $n$, and so $\gamma^{(r-t)i}$ can generate every power of $\gamma$. This would mean that should $\gamma^{(r-t)i}$ be a singleton element of the basis, then every power of $\gamma$ would be a singleton element of the basis also. However, $\gamma^r + \gamma^t$ is an element of $\mathcal{B}[\mathcal{S}_\Gamma]$, and so this would imply a contradiction. Therefore for every $i \in [0, \frac{n}{2}]$ which is relatively prime to $n$, we have $\gamma^{(r-t)i}+\gamma^{(t-r)i} \in \mathcal{B}[\mathcal{S}_\Gamma]$.

Let us recapitulate that at this point we have shown that if $n$ is the order of some $\gamma \in \Gamma$, $r$ and $t$ are two integers such that $(r-t)$ are relatively prime to $n$ and $(\gamma^r + \gamma^t)$ is in $\mathcal{B}[\mathcal{S}_\Gamma]$ then $\displaystyle\sum_{i = 0}^{\frac{n}{2}}[\binom{n}{i}(\gamma^{(r- t)i}+\gamma^{(t-r)i})]$ is in $\mathcal{S}_{\Gamma}$ and that those pairs of summands of the form $\gamma^{(r- t)i}+\gamma^{(t-r)i}$ appear as a basis element in $\mathcal{B}[\mathcal{S}_\Gamma]$ when $i$ is relatively prime to $n$. All that is left to show is that every number less than $n$ which is also relatively prime to it appears in the expression $\displaystyle\sum_{i = 0}^{\frac{n}{2}}[\binom{n}{i}(\gamma^{(r- t)i}+\gamma^{(t-r)i})]$ as a power of $\gamma$ (or, equivalently, a value of $(r-t)i$) and this will complete the proof.

Consider the fact that $\displaystyle\sum_{i = 0}^{\frac{n}{2}}[\binom{n}{i}(\gamma^{(r- t)i}+\gamma^{(t-r)i})] \equiv \displaystyle\sum_{i = 0}^{\frac{n}{2}}[\binom{n}{i}\gamma^{(r- t)i}] + \displaystyle\sum_{i = \frac{n}{2}+1}^{n}[\binom{n}{i}\gamma^{(r- t)i}]$.

Therefore, we can once again re-write our summation as $\displaystyle\sum_{i \in [n]}[\binom{n}{i} \gamma^{(r- t)i}]$ where $[n]$ is the set $\{0,1,2,\dots,n \}$, and re-write again to 

$\displaystyle\sum_{j \in (r-t)[n]}[\binom{n}{i} \gamma^{j}]$

where $(r-t)[n]$ is the set $\{(r-t)i$ mod $n : i \in [n]\}$.

However, we now note that the set $(r-t)[n]$ is in fact equivalent to $[n]$. This is because the modulo multiplication of a set can be represented as a permutation of a set unto itself, which is a bijective function. Therefore:

$\displaystyle\sum_{i \in [n]}[\binom{n}{i} \gamma^{(r- t)i}] \equiv \displaystyle\sum_{j \in [n]}[\binom{n}{i} \gamma^j]$

Obviously every number less than $n$ which is less relatively prime to it appears as a power of $\gamma$ on the right hand side above, meaning that those powers of $\gamma$ also appear on the left hand side above.

This means that every $\gamma^{(r-t)i}$ such that $(r-t)i$ is relatively prime to $n$ appears in the summation $\displaystyle\sum_{i = 0}^{\frac{n}{2}}[\binom{n}{i}(\gamma^{(r- t)i}+\gamma^{(t-r)i})]$ and as we stated above, this completes the proof.

\end{proof}

\begin{lems} \label{lm2}
Let $n$ be an odd integer and let $a$ and $b$ be the generators of the group $D_n$ with orders 2 and $n$ respectively. Also let $r$, $s$ and $t$ be unique integers less than or equal to $n$. The Schur ring $\langle\langle ab^r+ab^s+ab^t\rangle\rangle $ must be trivial if the following are true:

\begin{enumerate}

\item $ab^r+ab^s+ab^t \not \in \mathcal{B}[\langle\langle ab^r+ab^s+ab^t\rangle\rangle ]$
\item The absolute value of the difference between any two of $r$, $s$ and $t$ is relatively prime to $n$.
\item The sum of any two of $r$, $s$ and $t$ is not equal to twice the third variable taken modulo $n$.

\end{enumerate}
\end{lems}

As mentioned before, the proof for the second theorem will not be included here as it is included in the main work.

We can now begin on the proof of the main result. This proof is conceptually identical to the proof for the main theory in the main work.




\begin{proof}
We will prove Theorem \ref{thm:main} by showing how $ab^r+ab^s+ab^t$ cannot be an element of $\mathcal{B}[\langle \langle ab^r+ab^s+ab^t\rangle \rangle]$ and then use Lemma \ref{lm2} to meet the sufficient conditions of Theorem \ref{thm:trivial<<C>>} and so bring about the result.

We substitute $s$ with $4t-3r$ at this point for simplicity's sake. We will begin by assuming that $ab^r+ab^s+ab^{4t-3r}$ is in $\mathcal{B}[\langle \langle ab^r+ab^s+ab^{4t-3r}\rangle \rangle]$ and then show how this must imply a contradiction. We label this assumption as \textbf{\emph{(Asm)}} for further reference.

So let us begin by considering the following statement, which we label as \textbf{\emph{(i)}} for further reference:


\[(ab^r+ab^t+ab^{4t-3r})^2 = 3(1)+ b^{r-t}+b^{3(r-t)}+b^{4(r-t)}+b^{-(r-t)}+b^{-3(r-t)}+b^{-4(r-t)}\]

Assuming that \textbf{\emph{(Asm)}} is true then it must be possible to express \textbf{\emph{(i)}} as the sum of elements of $\mathcal{B}[\langle \langle ab^r+ab^s+ab^{4t-3r}\rangle \rangle]$. However, we shall show that doing so implies a contradiction which means that \textbf{\emph{(Asm)}} must be false.

We begin by noting that the group elements on the right hand side of \textbf{\emph{(i)}} must either be isolated in the basis of the Schur ring or be grouped with other elements also on the right hand side of \textbf{\emph{(i)}}.

For the sake of brevity let us denote the sum of the elements on the right hand side as $B_6$. We note now that $(B_6)^2$ is equal to


\begin{align*}
6(1) + 2B_6 + b^{6(r-t)}+b^{-6(r-t)}+b^{8(r-t)}+b^{-8(r-t)} + \\
2(b^{5(r-t)}+b^{-5(s-t)}+b^{7(r-t)}+b^{-7(r-t)})+3(b^{2(r-t)}+b^{-2(r-t)}).
\end{align*}

We notice that the total coefficient of $b^{2(r-t)}+b^{-2(r-t)}$ on the right hand side is 3 and so, by the Schur-Wielandt principle, either $b^{2(r-t)}$ and $b^{-2(r-t)}$ are isolated elements in the basis of the Schur ring or $b^{2(r-t)}+b^{-2(r-t)}$ is an element in the basis. Let us first consider the possibility that $b^{2(r-t)}$ is an isolated element.

Since $r-t$ is relatively prime to $n$ and $n$ is odd then $2(r-t)$ is also relatively prime to $n$. Therefore if $b^{2(r-t)}$ is isolated then all powers of $b$ are isolated. This implies that $b^{3(r-t)}$ must also be an isolated element. Since our assumption is that $ab^r+ab^s+ab^{4t-3r}$ is in $\mathcal{B}[\langle \langle ab^r+ab^s+ab^{3r-2s}\rangle \rangle]$ we can consider the product of $ab^r+ab^s+ab^{4t-3r}$ and $b^{3(r-t)}$, which must be in $\langle \langle ab^r+ab^s+ab^{4t-3r}\rangle \rangle$.


\[(ab^r+ab^t+ab^{4t-3r})b^{3(t-r)} \equiv ab^{3t-2r}+ab^{4t-3r}+ab^{7t-6r}.\]


Note how the term $ab^{4t-3r}$ appears above. Since we are assuming that $ab^r+ab^s+ab^{4t-3r}$ is in $\mathcal{B}[\langle \langle ab^r+ab^s+ab^{4t-3r}\rangle \rangle]$, $ab^r$ and $ab^t$ must appear too. There are only two ways this is possible:

\begin{itemize}
\item $3t-2r = t$ and $7t-6r = r$
\item $3t-2r = r$ and $7t-6r = s$
\end{itemize} 

All the equations above, however, imply that $r=t$ which is a contradiction. Therefore it is not possible that $b^{2(r-t)}$ is an isolated element. Let us consider now the second possibility that $b^{2(r-t)}+b^{-2(r-t)}$ is an element in the basis.

We note that the difference in powers between $b^{2(r-t)}$ and $b^{-2(r-t)}$ is $\pm4(r-t)$. Since $r-t$ is relatively prime to $n$ and $n$ is odd then $\pm4(r-t)$ is also relatively prime to $n$. Lemma \ref{lm1} therefore gives us that $b^{r-t} + b^{t-r}$ is in the basis of the Schur ring. Combining this with \textbf{\emph{(Asm)}} gives us that $(ab^r + ab^t + ab^{4t-3r})(b^{r-t}+b^{t-r})$ is in the Schur ring too. This expression expands to $ab^t +ab^r+ab^{2r-t}+ab^{2t-r}+ab^{3t-2r}+ ab^{5t-4r} $. Notably the terms $ab^t, ab^r$ appear here and due to \textbf{\emph{(Asm)}} this means that $ab^{4t-3r}$ must appear also. There are three ways this could happen:

\begin{itemize}
\item $2r-t=4t-3r$ gives us $r = t$ 
\item $2t-r=4t-3r$ gives us $r = t$ 
\item $3t-2r=4t-3r$ gives us $r = t$
\item $5t-4r=4t-3r$ gives us $r = t$
\end{itemize}

But all these ways imply $r=t$ which is a contradiction. Therefore it is also impossible that $b^{2(r-t)}+b^{-2(r-t)}$ is an element in the basis.

So we have seen that due to  \textbf{\emph{(Asm)}} either $b^{2(r-t)}$ and $b^{-2(r-t)}$ are isolated elements in the basis of the Schur ring or $b^{2(r-t)}+b^{-2(r-t)}$ is an element in the basis. However, we have also seen that both these possibilities imply contradictions. Therefore it must be the case that  \textbf{\emph{(Asm)}} is false.

This means that $ab^r+ab^s+ab^{3r-2s} \not \in \mathcal{B}[\langle\langle ab^r+ab^s+ab^{3r-2s}\rangle\rangle ]$. Lemma 2.2 gives us that this implies that $\langle\langle ab^r+ab^s+ab^t\rangle\rangle $ is trivial and so $\Aut(\langle \langle ab^r+ab^s+ab^t \rangle \rangle) \equiv D_n$ meaning that $\Cay(D_n, \{ ab^r,ab^s,ab^t \})$ is a GRR of $D_n$ as we saw in the first section.

Hence we have proven the theorem.
\end{proof}

\vspace{8mm}

\bibliographystyle{plain}

\begin{thebibliography}{9}

 \bibitem{ABCFHMST}
  M. Albert, J. Bratz, P. Cahn, T. Fargus, N. Haber, E. Mcmahon, J. Smith, S. Tekansik,
  \textit{Color-Permuting Automorphisms of Cayley Graphs},
  Congressus Numerantium \textbf{Vol. 190}, pages 161 - 171
  (2008).
  
  \bibitem{AlNa}
  A.V. Alexeevski, S.M. Natanzon,
  \textit{Algebras of conjugacy classes of partial elements},
  American Mathematical Society Translations \textbf{Vol. 234}, pages 1 - 12
  (2014).
  
  \bibitem{LuIrMiDm}
  L. Babel, I.V. Chuvaeva, M. Klin, D.V. Pasechnik,  
  \textit{Algebraic Combinatorics in Mathematical Chemistry.Methods and Algorithms.II. Program Implementation of the Weisfeiler-Leman Algorithm},
  arXiv 1002.1921
  (2010).
  
  \bibitem{BeWi}
  L.W. Beineke, R.J. Wilson (ed.),
  \textit{Topics in Algebraic Graph Theory},
  Cambridge University Press, UK
  (2004).
  
  \bibitem{Bi}
  N. Biggs,
  \textit{Algebraic Graph Theory Second Edition},
  Cambridge Univeristy Press, UK
  (1993).
  
  \bibitem{Ch}
  C.Y. Chao,
  \textit{On Groups and Graphs},
  American Mathematical Society \textbf{Vol. 118}, pages 488 - 497
  (1965).
  
  \bibitem{Cox}
  H.S.M. Coxeter, R. Frucht, D.L. Powers,
  \textit{Zero-Symmetric Graphs - Trivalent Graphical Regular Representations of Groups},
  Academic Press Inc., New York
  (1981).
  
  \bibitem{De}
  P. Delsarte,
  \textit{An Algebraic Approach to the Association Schemes of Coding Theory},
  Philips Research Report Supplements \textbf{Issue 10},
  (1973).
  
   \bibitem{DrGiGo}
  M. Droste, M. Giraudet, R. Gobel,
  \textit{All Groups are Outer Automorphism Groups of Simple Groups},
  Journal of London Mathematical Society (2) \textbf{Vol. 64}, pages 565 - 575
  (2001).

  \bibitem{Eb&La}
J.Ebejer, J.Lauri
\textit{Using Schur Rings to Produce GRRs for Dihedral Groups.}
arXiv preprint arXiv:2401.02305 
(2024).
  
  \bibitem{Eb}
  J.Ebejer, 
  \textit{Schur rings: some theory and applications to graphical regular representations.}
  MSc dissertation, University of Malta
  (2020).
  
  \bibitem{FaIvKl}
  I.A. Farad\v{z}ev, A.A. Ivanov, M.H. Klin,
  \textit{Galois  Correspondence  Between Permutation  Groups  and  Cellular  Rings  (Association  Schemes)},
  Graphs and Combinatorics \textbf{Vol. 6}, pages 303 - 332
  (1990).
  
  \bibitem{Go}
  C.D. Godsil,
  \textit{Neighbourhoods of transitive graphs and GRRs},
  Journal of Combinatorial Theory (Ser. B) \textbf{Vol. 29}, pages 116 - 140
  (1980).
  
  \bibitem{He}
  D. Hetzel,
  \textit{{\"U}ber regul{\"a}re graphische Darstellungen von auﬂ{\"o}sbaren Gruppe},
  Diplomarbeit. Technische Universit{\"a}t Berlin
  (1976).
  
   \bibitem{Im}
  W. Imrich,
  \textit{Graphical representations of groups of odd order},
  Combinatorics Colloq. Math. Soc. J. Bolyai \textbf{Vol. 18}, pages 611 - 621
  (1976).
  
  \bibitem{Im2}
  W. Imrich,
  \textit{On graphs with regular groups},
  Journal of Combinatorial Theory (Ser. B) \textbf{Vol. 19}, pages 174 - 180
  (1975).
  
  \bibitem{ImWa}
  W. Imrich, M.E. Watkins,
  \textit{On graphical regular representation of cyclic extensions of groups},
  Pacific Journal of Mathematics \textbf{Vol. 55}, pages 461 - 477
  (1974).
  
  \bibitem{ImWa2}
  W. Imrich, M.E. Watkins,
  \textit{On automorphism groups of Cayley graphs},
  Periodica Mathematica Hungarica \textbf{Vol. 7}, pages 243 - 258
  (1976).
  
  \bibitem{IvKe}
  V.N. Ivanov, S.V. Kerov,
  \textit{The Algebra of Conjugacy Classes in Symmetric Groups and Partial Permutations},
  Journal of Mathematical Sciences \textbf{Vol. 107}, pages 4212 - 4230
  (2001).
  
  \bibitem{KlRuRuTi}
  M. Klin, C. Rucker, G. Rucker, G. Tinhofer,
  \textit{Algebraic Combinatorics in Mathematical Chemistry. Methods and Algorithms. I. Permutation Groups and Coherent (Cellular) Algebras},
  MATCH \textbf{Vol. 40}, pages 7 - 138
  (1999).
  
  \bibitem{LaSc}
  J. Lauri, R. Scapellato,
  \textit{Topics in graph automorphisms and reconstruction},
  Cambridge University Press
  (2016).
  
  \bibitem{MuPo}
  M. Muzychuk, I. Ponomarenko,
  \textit{Schur Rings},
  European Journal of Combinatorics \textbf{Vol. 30}, pages 1526 - 1539
  (2009).
  
  \bibitem{NoWa}
  L.A. Nowitz, M.E. Watkins,
  \textit{Graphical regular representations of non-abelian groups, I},
  Canadian Journal of Mathematics XXIV. \textbf{Vol. 6}, pages 993 - 1008
  (1972).
  
  \bibitem{NoWa2}
  L.A. Nowitz, M.E. Watkins,
  \textit{Graphical regular representations of non-abelian groups, II},
  Canadian Journal of Mathematics XXIV. \textbf{Vol. 6}, pages 1009 - 1018
  (1972).
  
  \bibitem{Sa}
  G. Sabidussi,
  \textit{On a Class of fixed-point free graphs},
  Proceedings of the American Mathematical Society \textbf{Vol. 9}, pages 800 - 804
  (1958).  
  
  \bibitem{Wei}
  B. Weisfeiler,
  \textit{On Construction and Identification of Graphs},
  Springer-Verlag, Berlin
  (1976).
   
  
  \bibitem{Wi}
  H. Wielandt,
  \textit{Finite Permutation Groups},
  Academic Press, New York-London
  (1964).
  
  
\end{thebibliography}

\end{document}